\documentclass{birkjour}
\usepackage[utf8]{inputenc}
\usepackage{amssymb,mathtools,hyperref,cases,enumerate}
\usepackage[shortlabels]{enumitem}

\usepackage{xcolor}

\usepackage[backend=biber]{biblatex}
\AtEveryBibitem{\clearlist{language}}
\numberwithin{equation}{section}

\newtheorem{thm}{Theorem}[section]
\newtheorem{cor}[thm]{Corollary}

\newtheorem{lem}[thm]{Lemma}
\theoremstyle{remark}
\newtheorem*{claim}{Claim}
\newtheorem{rmk}[thm]{Remark}
\newtheorem*{question}{Question}
\theoremstyle{definition}
\newtheorem{defn}[thm]{Definition}

\newcommand{\eps}{\varepsilon}

\newcommand{\R}{\mathcal{R}}

\newcommand{\dir}{\mathrm{dir}}
\newcommand{\rad}{\mathrm{rad}}
\newcommand{\azi}{\mathrm{azi}}
\newcommand{\pol}{\mathrm{pol}}

\DeclareMathOperator{\vspan}{span}
\DeclareMathOperator{\id}{id}

\newcommand{\nat}{\mathbb{N}}
\newcommand{\real}{\mathbb{R}}

\newcommand{\dif}{\mathrm{d}}
\newcommand{\End}{\mathcal{L}}

\DeclarePairedDelimiter{\abs}{\lvert}{\rvert}
\DeclarePairedDelimiter{\norm}{\lVert}{\rVert}
\DeclarePairedDelimiter{\parens}{(}{)}
\DeclarePairedDelimiter{\set}{\{}{\}}
\DeclarePairedDelimiter{\brackets}{\lbrack}{\rbrack}

\DeclarePairedDelimiter{\angles}{\langle}{\rangle}
\DeclarePairedDelimiter{\cci}{\lbrack}{\rbrack}
\DeclarePairedDelimiter{\coi}{\lbrack}{\lbrack}
\DeclarePairedDelimiter{\oci}{\rbrack}{\rbrack}
\DeclarePairedDelimiter{\ooi}{\rbrack}{\lbrack}

\title[Cluster solutions to the Schrödinger--Bopp--Podolsky system]{Cluster semiclassical states of the nonlinear Schrödinger--Bopp--Podolsky system}
\begin{document}

\author{Gustavo de Paula Ramos}
\address{Instituto de Matemática e Estatística, Universidade de São Paulo, Rua do Matão, 1010, 05508-090 São Paulo SP, Brazil}
\email{gpramos@ime.usp.br}
\urladdr{http://gpramos.com}

\begin{abstract}
Consider the following nonlinear Schrödinger--Bopp--Podolsky system in $\real^3$:
\[
\begin{cases}
-\eps^2 \Delta u + \parens{V + \phi} u = u \abs{u}^{p-1};
\\
a^2 \Delta^2 \phi - \Delta \phi = 4 \pi u^2,
\end{cases}
\]
where $a, \eps > 0$; $1 < p < 5$;
$V \colon \real^3 \to \ooi{0, \infty}$ and we want to solve for
$u, \phi \colon \real^3 \to \real$. By means of Lyapunov--Schmidt reduction, we show that if $K \geq 2$, $z_0$ is a strict local minimum of $V$, $V$ is adequately flat in a neighborhood of $z_0$ and $\eps$ is sufficiently small, then the system has a multipeak cluster solution with $K$ peaks placed at the vertices of a regular convex $K$-gon centered at $z_0$.


\end{abstract}

\date{\today}
\maketitle
\tableofcontents

\section{Introduction}

This note is concerned with the existence of multipeak cluster solutions to the following \emph{nonlinear Schrödinger--Bopp--Podolsky (SBP) system} in $\real^3$:
\begin{equation}
\label{intro:eqn:SBP}
\begin{cases}
-\eps^2 \Delta u + \parens{V + \phi} u = u \abs{u}^{p - 1};\\
a^2 \Delta^2 \phi - \Delta \phi = 4 \pi u^2,
\end{cases}
\end{equation}
where $\eps > 0$ is sufficiently small, $a > 0$, $1 < p < 5$,
$V \colon \real^3 \to \ooi{0, \infty}$ and we want to solve for
$u, \phi \colon \real^3 \to \real$.

By formally replacing $a = 0$ in \eqref{intro:eqn:SBP}, we recover the well-studied \emph{nonlinear Schrödinger--Poisson system},
\begin{equation}
\label{intro:eqn:SP}
\begin{cases}
-\eps^2 \Delta u + \parens{V + \phi} u = u \abs{u}^{p - 1};\\
- \Delta \phi = 4 \pi u^2,
\end{cases}
\end{equation}
which seems to have been first considered by D'Aprile and Mugnai in \cite{daprileNonExistenceResultsCoupled2004a} and provides the most commonly considered model for coupling the nonlinear Schrödinger equation with the electrostatic self-interaction of charged quantum matter.

The nonlinear SBP system was introduced in the mathematical literature in \cite{daveniaNonlinearSchrodingerEquation2019}, where d'Avenia \& Siciliano established the existence/non-existence of solutions to
\[
\begin{cases}
	-\Delta u+\omega u + q^2 \phi u = u \abs{u}^{p - 1};\\
	a^2 \Delta^2 \phi - \Delta \phi = 4 \pi u^2
\end{cases}
\]
for fixed $a, \omega > 0$ in function of the values of $p$, $q$ and proceeded to prove that, in the radial case, its solutions tend to solutions of the corresponding Schrödinger--Poisson system as
$a \to 0$. Since then, there has been an increasing interest for the SBP system. Let us recall a few developments for related non-autonomous systems in $\real^3$. Results regarding existence and multiplicity may be found in \cite{chenGroundStateSolutions2021, depaularamosConcentratedSolutionsSchrodinger2024, figueiredoMultipleSolutionsSchrodinger2023, liuPositiveSolutionsNonautonomous2023, mascaroPositiveSolutionsSchrodingerbopppodolsky2022, pengExistenceMultiplicitySolutions2022a, tengExistencePositiveBound2021a}; the existence of solutions for critical nonlinearities has already attracted a lot of attention, being considered in \cite{chenCriticalSchrodingerBopp2020, liGroundStateSolutions2020a, liuExistenceAsymptoticBehaviour2022, yangExistenceNontrivialSolution2020, zhuSchrodingerBoppPodolsky2021} and the existence and multiplicity of sign-changing solutions were studied in \cite{huExistenceLeastenergySignchanging2022, wangExistenceMultiplicitySignchanging2022, zhangSignchangingSolutionsClass2024, zhangSignchangingSolutionsSchrodinger2022}.

It was only recently that the existence of multipeak solutions was addressed in the literature. In \cite{wangMultibumpSolutionsSchrodinger2024}, Wang, Wang \& Wang proved that the following SBP system in $\real^3$ admits multipeak solutions:
\begin{equation}
\label{intro:eqn:WangWangWang}
\begin{cases}
- \Delta u + \parens{\lambda V + V_0} u + K \phi u
=
u \abs{u}^{p - 2};\\
a^2 \Delta^2 \phi - \Delta \phi = K u^2,
\end{cases}
\end{equation}
where $4 < p < 6$; $V_0 \in C \parens{\real^3}$;
$V \in C \parens{\real^3, \coi{0, \infty}}$ and
$K \colon \real^3 \to \coi{0, \infty}$. More precisely, they employed the penalization method to show that if $\lambda > 0$ is sufficiently large and $\Omega := \mathrm{int} ~ V^{- 1} \parens{0}$ is a nonempty bounded set with smooth boundary and $m$ connected components, then \eqref{intro:eqn:WangWangWang} has multipeak solutions with $K \leq m$ peaks, each of them in a different connected component of $\Omega$ and the corresponding solutions become increasingly concentrated as
$\lambda \to \infty$.

Our goal in this note is to consider the question of whether a different kind of multipeak solutions exists, that is:
\begin{question}
Does \eqref{intro:eqn:SBP} admit a family of multipeak solutions clustered around a local strict minimum point of $V$?
\end{question}

By means of Lyapunov--Schmidt reduction, we will answer this question positively under technical hypotheses on $V$ which assure that the potential is adequately flat around $z_0$. We remark that, besides a recent publication by the author (\cite{depaularamosConcentratedSolutionsSchrodinger2024}), we are unaware of other articles which employed Lyapunov--Schmidt reduction in the context of the nonlinear SBP system.

Let us develop the preliminaries to state our main result. We begin by introducing our hypotheses on $V$.
\begin{enumerate}
[label=(V$_\arabic*$)]
\item \label{V_1}
We have
$
1
=
V \parens{0}
=
\min V|_{\overline{B_1 \parens{0}}}
<
V \parens{x}
$
for every $x \in \overline{B_1 \parens{0}} \setminus \set{0}$;
\item \label{V_2}
there exist $\alpha > 3 + \sqrt{7}$ and
$g \in C^{3, 1} \parens{\overline{B_1 \parens{0}}, \coi{0, \infty}}$ such that
\begin{itemize}
\item
$\norm{g}_{C^{3, 1}} \leq 1$,
\item
$g'' \parens{0}$ is positive-definite and
\item
$V \parens{x} = 1 + g \parens{x}^\alpha$
for every $x \in \overline{B_1 \parens{0}}$;
\end{itemize}
\item \label{V_3}
$0 < \inf V \leq \sup V < \infty$.
\end{enumerate}
Hypotheses \ref{V_1}, \ref{V_2} are stated in this form to simplify the notation, being clear that they could be respectively replaced with the following hypotheses without substantial changes to the arguments (see \cite[Section 8.3]{ambrosettiNonlinearAnalysisSemilinear2007}):
\begin{enumerate}
[label=(V$'_\arabic*$)]
\item
there exist $z_0 \in \real^3$ and $\rho > 0$ such that
$
V \parens{z_0}
=
\min V|_{\overline{B_\rho \parens{z_0}}}
<
V \parens{x}
$
for every $x \in \overline{B_\rho \parens{z_0}} \setminus \set{z_0}$;
\item
there exist $\alpha > 3 + \sqrt{7}$ and
$
g
\in
C^{3, 1} \parens{\overline{B_\rho \parens{z_0}}, \coi{0, \infty}}
$
such that
\begin{itemize}
\item
$g'' \parens{z_0}$ is positive-definite and
\item
$V \parens{x} = V \parens{z_0} + g \parens{x}^\alpha$
for every $x \in \overline{B_\rho \parens{z_0}}$.
\end{itemize}
\end{enumerate}

In Section \ref{variational}, we show that critical points of the \emph{energy functional} given by
\[
I_\eps \parens{u}
:=
\frac{1}{2}
\int \parens*{\abs{\nabla u}^2 + V_\eps u^2}
+
\frac{\eps^3}{4} \int \parens{\phi_{\eps, u^2} u^2}	
-
\frac{1}{p+1} \norm{u}_{L^{p+1}}^{p+1}
\]
for every $u \in H^1$ are naturally associated with solutions to \eqref{intro:eqn:SBP}, where
\begin{itemize}
\item
$\phi_{\eps, u w} := \kappa \parens{\eps \cdot} \ast \parens{u w}$ for every $u, w \in H^1$ and
\item
$\kappa \parens{x} := \parens{1 - e^{- \frac{\abs{x}}{a}}} / \abs{x}$
for every $x \in \real^3 \setminus \set{0}$.
\end{itemize}
As such, we will most often refer to the critical point equation
$\nabla I_\eps \parens{u} = 0$ instead of \eqref{intro:eqn:SBP}.

We are interested in solutions whose asymptotic profile is given by the sum of translations of the function
$U \colon \real^3 \to \ooi{0, \infty}$ defined as the unique positive radial solution to
\begin{equation}
\label{intro:eqn:NLSE}
\begin{cases}
-\Delta u+u=u\abs{u}^{p-1}	&\text{in}~\real^3;\\
u \parens{x} \to 0			&\text{as}~\abs{x} \to \infty
\end{cases}
\end{equation}
(see the theorem in \cite[p. 23]{kwongUniquenessPositiveSolutions1989}). It will be useful to consider translations of $U$ in function of spherical coordinates as follows:
\[
U_{r, \theta, \varphi} \parens{x}
:=
U \parens*{x - P \parens{r, \theta, \varphi}},
\]
where $P \colon \ooi{0, \infty} \times \real \times \real \to \real^3$
is given by
\[
P \parens{r, \theta, \varphi}
:=
r \parens{
	\cos \theta \sin \varphi,
	\sin \theta \sin \varphi,
	\cos \varphi
}
=
r P \parens{1, \theta, \varphi}
\in \real^3.
\]

We proceed to a statement of our main result.
\begin{thm}
\label{intro:thm}
Given $K \geq 2$, there exist $\eps_K > 0$,
\[
\set{
	w_\eps \in H^1: \nabla I_\eps \parens{w_\eps} = 0
}_{0 < \eps < \eps_K}
\]
and
\[
\set*{
	\parens{r_\eps, \theta_\eps, \varphi_\eps}
	\in
	\ooi{0, \infty} \times \coi{0, 2 \pi} \times \cci{0, \pi}
}_{0 < \eps < \eps_K} 
\]
such that
$\norm{w_\eps - W_{r_\eps, \theta_\eps, \varphi_\eps}}_{H^1} \to 0;$
$\eps r_\eps \to 0$ and $r_\eps \to \infty$ as $\eps \to 0$, where
\[
W_{r_\eps, \theta_\eps, \varphi_\eps}
:=
\sum_{1 \leq j \leq K}
	U_{
		r_\eps,
		\theta_\eps,
		\parens*{\varphi_\eps + \frac{2 \pi}{K} j}
	}.
\]
\end{thm}

We recall that due to \cite[Lemma 2 in p. 329]{berestyckiNonlinearScalarField1983a}, $U$ and $\abs{\nabla U}$ have exponential decay at infinity, i.e., there exists $\eta > 0$  such that
\begin{equation}
\label{intro:eqn:exponential_decay}
U \parens{x}, \abs{\nabla U \parens{x}} \lesssim e^{- \eta \abs{x}}
\quad \quad \text{for every} \quad \quad
x \in \real^3.
\end{equation}
As such, we deduce that given $z \in \real^3$, $U \parens{\cdot - z}$ is a critical point of the functional $I_0 \colon H^1 \to \real$ given by
\[
I_0 \parens{u}
=
\frac{1}{2} \norm{u}_{H^1}^2
-
\frac{1}{p + 1} \norm{u}_{L^{p + 1}}^{p + 1}.
\]

Our strategy to construct solutions to \eqref{intro:eqn:SBP} consists in interpreting $I_\eps$ as a perturbation of $I_0$. As such, we look for critical points of $I_\eps$ obtained as perturbations of
$
\sum_{1 \leq j \leq K}
U_{r, \theta, \parens*{\varphi + \frac{2 \pi}{N} j}}
$
for sufficiently large $r > 0$, which are close enough to critical points of $I_\eps$ due to the exponential decay of $U$ and the fact that \eqref{intro:eqn:NLSE} is invariant by translation (see Lemma \ref{pseudo:lem:estimate}).

Lyapunov--Schmidt reduction has been already employed in \cite{ruizClusterSolutionsSchrodingerPoissonSlater2011, ianniNonradialSignchangingSolutions2015} to construct multipeak cluster solutions to the nonlinear Schrödinger--Poisson system \eqref{intro:eqn:SP} (see also \cite{ianniConcentrationPositiveBound2008} for single-peaked solutions), so let us compare the technical details in the contexts of these two systems, starting with the hypotheses on $V$. Hypotheses \ref{V_1}, \ref{V_3} correspond to \cite[(V1) and (V3)]{ruizClusterSolutionsSchrodingerPoissonSlater2011}. Now, consider \ref{V_2}. We recall that critical points of the functional
$J_\eps \colon H^1 \to \real$ given by
\[
J_\eps \parens{u}
:=
\frac{1}{2} \int \parens*{\abs{\nabla u}^2 + V_\eps u^2}
+
\frac{\eps^2}{4}
\int \int
	\frac{u \parens{x}^2 u \parens{y}^2}{\abs{x - y}}
\dif x \dif y
-
\frac{1}{p+1} \norm{u}_{L^{p+1}}^{p+1}
\]
are naturally associated with solutions to \eqref{intro:eqn:SP}. The higher regularity of $g$ in \ref{V_2} when compared with \cite[(V2)]{ruizClusterSolutionsSchrodingerPoissonSlater2011} is related to the fact that $I_\eps$ is a perturbation of $I_0$ of order
$\eps^3$, while $J_\eps$ is a perturbation of order $\eps^2$. The inequality $\alpha > 3 + \sqrt{7}$ is similarly derived from technical constraints of the problem (see Remark \ref{pseudo:rmk}).

We highlight that our main result does not follow from a straightforward adaptation of the arguments in \cite{ruizClusterSolutionsSchrodingerPoissonSlater2011}, where Ruiz \& Vaira prove the existence of multipeak cluster solutions without considering pseudo-critical points of $J_\eps$ whose peaks are disposed in a specific geometric configuration.

Such an argument is only possible in the framework of Maxwell's electromagnetism (for the precise mathematical description of the phenomenon, see Remark \ref{reduced:rmk:impossible}). Indeed, consider a system with $K$ charged particles repulsing each other through electrostatic interaction. The Coulomb energy of two particles which are sufficiently close one to the other surpasses the Coulomb energy of the interaction involving any of the remaining particles (this phenomenon is usually called the \emph{infinity problem}), and this fact is exploited in the proof of \cite[Proposition 4.4]{ruizClusterSolutionsSchrodingerPoissonSlater2011} when Ruiz \& Vaira consider $\mathbf{P}^\eps$ such that
$
\abs{P_i^\eps - P_j^\eps}
=
\eps^{\delta + \frac{2 - \alpha}{\alpha + 1}}
$.
This does not happen in the Bopp--Podolsky electromagnetism because the energy of interaction between two particles is finite. As such, we can only establish the existence of multipeak cluster solutions by restricting our search to pseudo-critical points of $I_\eps$ whose peaks are disposed in specific geometric configurations, arguing similarly as in \cite{ianniNonradialSignchangingSolutions2015}.

Let us comment on the organization of the text. In Section \ref{variational}, we introduce the relevant variational framework, while we perform the Lyapunov--Schmidt reduction and we prove Theorem \ref{intro:thm} in Section \ref{existence}.

\subsection*{Notation}

Brackets are exclusively employed to enclose the argument of (multi-)linear functions. Given a set $X$ and
$f,g \colon X \to \coi{0, \infty}$, we write
\[
f \parens{x} \lesssim g \parens{x}
\quad \quad \text{for every} \quad \quad
x \in X
\]
when there exists $C > 0$ such that $f \parens{x} \leq C g \parens{x}$ for every $x \in X$. We let $H^1_\eps$ denote the Hilbert space obtained as completion of $C_c^\infty$ with respect to
\[
\angles{u_1 \mid u_2}_{H^1_\eps}
:=
\int\parens{\nabla u_1 \cdot \nabla u_2 + V_\eps u_1 u_2}.
\]
The natural inclusions
$H^1_\eps \hookrightarrow H^1$ and $H^1 \hookrightarrow H^1_\eps$
are continuous due to \ref{V_3}, so $H^1$ and $H^1_\eps$ are canonically isomorphic as Hilbert spaces.

\subsection*{Acknowledgement}

The author thanks João Fernando da Cunha Nariyoshi for conversations about the problem considered in this text. This study was financed in part by the Coordenação de Aperfeiçoamento de Pessoal de Nível Superior - Brasil (CAPES) - Finance Code 001.

\section{Variational framework}
\label{variational}

Given $u_1, u_2 \in H^1$, we say that $\phi$ is a \emph{weak solution} to
\begin{equation}
\label{variational:eqn:BP}
\Delta^2 \phi - \Delta \phi = 4 \pi u_1 u_2
\end{equation}
when $\phi \in X$ and
$\angles{\phi, w}_{X} = 4 \pi \int \parens{u_1 u_2 w}$
for every $w \in C_c^\infty$, where $X$ denotes the Hilbert space obtained as completion of $C_c^\infty$ with respect to
\[
\angles{u_1, u_2}_X
:=
\int\parens{\Delta u_1 \Delta u_2 + \nabla u_1 \cdot \nabla u_2}.
\]
It follows from the Riesz Representation Theorem that \eqref{variational:eqn:BP} has a unique weak solution and we even know its explicit expression,
$\phi_{1, u_1 u_2} = \kappa \ast \parens{u_1 u_2} \in X$ (see \cite[Section 3.1]{daveniaNonlinearSchrodingerEquation2019}).

In this context, $\parens{u, \varphi}$ is said to be a \emph{weak solution} to \eqref{intro:eqn:SBP} when $u \in H^1$; $\varphi = \phi_{1, u^2} \in X$ and
\[
\int \parens*{
	\eps^2 \nabla u \cdot \nabla w + \parens{V + \phi_{1, u^2}} u w
}
=
\int \parens{u \abs{u}^{p - 1} w}
\]
for every $w \in C_c^\infty$. Through the change of variable
$x \mapsto \eps x$, this condition is rewritten as
\begin{equation}
\label{variational:eqn:weak_sol_SBP'}
\int \parens*{
	\nabla u \cdot \nabla w
	+
	\parens{V_\eps + \eps^3 \phi_{\eps, u^2}} u w
}
=
\int \parens{u \abs{u}^{p - 1} w}
\quad
\text{for every}~w \in C_c^\infty.
\end{equation}
where $\phi_{\eps, u_1 u_2} := \kappa_\eps \ast \parens{u_1 u_2}$ and $\kappa_\eps := \kappa \parens{\eps \cdot}$.

We highlight the following computational properties:
\[
\parens{H^1}^4 \ni \parens{u_1, u_2, u_3, u_4}
\mapsto
\int \parens{\phi_{\eps, u_1 u_2} u_3 u_4}
=
\int \parens{\phi_{\eps, u_3 u_4} u_1 u_2}
\in
\real
\]
is multilinear and
\[
\abs*{\int \parens{\phi_{\eps, u_1 u_2} u_3 u_4}}
\lesssim
\norm{u_1}_{L^2} \norm{u_2}_{L^2} \norm{u_3}_{L^2} \norm{u_4}_{L^2}
\]
for every $u_1, u_2, u_3, u_4 \in H^1$ and $\eps > 0$.

The facts that $I_\eps$ is well-defined and of class $C^2$ follow from the previous paragraph. More precisely,
\[
I_\eps' \parens{u} \brackets{w_1}
=
\angles{u, w_1}_{H^1_\eps}
+
\eps^3 \int \parens{\phi_{\eps,u^2} u w_1}
-
\int\parens{u \abs{u}^{p-1} w_1}
\]
and
\begin{multline*}
I_\eps'' \parens{u} \brackets{w_1,w_2}
=
\angles{w_1, w_2}_{H^1_\eps}
+
\eps^3 \int \parens{
	\phi_{\eps, u^2} w_1 w_2 + 2 \phi_{\eps, u w_1} u w_2
}
+
\\
-
p \int \parens*{\abs{u}^{p-1} w_1 w_2}
\end{multline*}
for every $u, w_1, w_2 \in H^1$. The following variational characterization then becomes immediate:
\[
\parens{u, \phi_{\eps, u^2}} \in H^1 \times X
\quad \text{satisfies} \quad \eqref{variational:eqn:weak_sol_SBP'}
\iff
\nabla I_\eps \parens{u} = 0. 
\]
Notice that these weak solutions are classical solutions due to regularity theory (see \cite[Appendix A.1]{daveniaNonlinearSchrodingerEquation2019}).

\section{Existence of cluster solutions}
\label{existence}

The goal of this section is to prove Theorem \ref{intro:thm}. As such, we consider a fixed $K \geq 2$ and we omit the dependence of any quantities on it. Let us introduce a few additional notational conventions. We define
\begin{equation}
\label{existence:eqn:translation_of_varphi}
\varphi_j = \varphi + \frac{2 \pi}{K} j
\end{equation}
for every $j \in \set{1, \ldots, K}$ and $\varphi \in \real$. Given $j, k \in \set{1, \ldots, K}$, we let
\[
d_{j, k}
=
P \parens*{1, 0, \frac{2 \pi}{N} j}
-
P \parens*{1, 0, \frac{2 \pi}{N} k}
\in \real^3,\]
so that
$
d_{j,k} = P \parens{1, \theta, \varphi_j} - P \parens{1, \theta, \varphi_k}
$
for every $\theta, \varphi \in \real$ and, in particular,
\[
\abs{d_{j,k}}
=
\sqrt{2 - 2 \cos \parens*{\frac{2 \pi}{K} \parens{j - k}}}.
\]

\subsection{Pseudo-critical points of $I_\eps$}

Our first goal is to obtain an important decomposition of the critical point equation
\begin{equation}
\label{equivalent:eqn:critical_point}
\nabla I_\eps \parens{u} = 0; \quad \quad u \in H^1
\end{equation}
as a system of equations. Let
\[
\dot{U}^{\rad}_{r, \theta, \varphi}
=
\left.
	\frac{\dif}{\dif R} U_{R, \theta, \varphi}
\right|_{R = r};
~
\dot{U}^{\azi}_{r, \theta, \varphi}
=
\left.
	\frac{\dif}{\dif \Theta} U_{r, \Theta, \varphi}
\right|_{\Theta = \theta};
~
\dot{U}^{\pol}_{r, \theta, \varphi}
=
\left.
	\frac{\dif}{\dif \Phi} U_{r, \theta, \Phi}
\right|_{\Phi = \varphi}
\]
and similarly for $W_{r, \theta, \varphi}$. We use these functions to introduce the orthogonal decomposition
\begin{equation}
\label{pseudo:eqn:decomposition_of_H^1}
H^1 = N_{r, \theta, \varphi} \oplus T_{r, \theta, \varphi},
\end{equation}
where
\[
T_{r, \theta, \varphi}
:=
\vspan \set*{
	\dot{W}^{\rad}_{r, \theta, \varphi}; ~
	\dot{W}^{\azi}_{r, \theta, \varphi}; ~
	\dot{W}^{\pol}_{r, \theta, \varphi}
} \subset H^1
\quad \quad \text{and} \quad \quad
N_{r, \theta, \varphi}
:=
T_{r, \theta, \varphi}^\perp.
\]
As such, we rewrite \eqref{equivalent:eqn:critical_point} as the following system of equations induced by the orthogonal decomposition \eqref{pseudo:eqn:decomposition_of_H^1}:
\[
\begin{cases}
\Pi_{r, \theta, \varphi} \brackets{\nabla I_\eps \parens{u}}
=
0
&\text{\emph{(auxiliary equation)}};
\\
\parens{\id_{H^1} - \Pi_{r, \theta, \varphi}}
\brackets{\nabla I_\eps \parens{u}}
=
0
&\text{\emph{(bifurcation equation)}};
\\
u\in H^1,
\end{cases}
\]
where
$\Pi_{r, \theta, \varphi} \colon H^1 \to N_{r, \theta, \varphi}$ denotes the $H^1$-orthogonal projection.

We are interested in solutions to \eqref{equivalent:eqn:critical_point} according to the \emph{ansatz}
\[
u = W_{r, \theta, \varphi} + n,
\quad \quad \text{where} \quad \quad
n \in N_{r, \theta, \varphi}.
\]
In fact, we will only consider $r > 0$ in a set of admissible distances to $0$ from the peaks of the pseudo-critical points of $I_\eps$.

Before defining this set, we need to fix two quantities. Notice that
\[
0
<
2 \frac{\alpha + 2}{\alpha - 1}
<
\min \parens{2 \alpha - 8, \alpha}
\quad \quad \iff \quad \quad
\alpha > 3 + \sqrt{7},
\]
so \ref{V_2} assures that we can fix
\[
\lambda
\in
\ooi*{
	2\frac{\alpha + 2}{\alpha - 1},
	\min \parens{2 \alpha - 8, \alpha}
}
\quad \quad \text{and} \quad \quad
\beta \in \ooi*{
	0, \frac{\alpha - \lambda}{\alpha + 1}
}.
\]
Let us comment on the technical reasons for these choices.
\begin{rmk}
\label{pseudo:rmk}
We need the inequality
$\lambda > 2 \parens{\alpha + 2} / \parens{\alpha - 1}$ to treat Case \ref{Case2} in the proof of Lemma \ref{reduced_functional:lem:existence_of_minimum}. As such, we obtain $\lambda > 2$ for any $\alpha > 1$. We also need
$\lambda < \min \parens{2 \alpha - 8, 3 \alpha - 10}$
(which follows from $\lambda < 2 \alpha - 8$ when $\alpha \geq 2$)
to obtain an appropriate bound for the error in Lemma \ref{reduced:lem:expansion_of_Ieps}.
\end{rmk}

As such, we define the set of \emph{admissible distances} to 0 as
\[
R_\eps = \set*{
	r > 0:
	\frac{\eps^{\beta}}{\eps^{\frac{\alpha - \lambda}{\alpha + 1}}}
	<
	r
	<
	\frac{1}{\eps^{\frac{\alpha}{\alpha + 1}}}
	\quad \text{and} \quad
	\max_{\abs{x} = 1} V \parens{\eps r x}
	<
	1 + \eps^{\frac{3 \alpha}{\alpha + 1}}
}.
\]
It is easy to check that
$\eps^{- \frac{\alpha - \lambda}{\alpha + 1}} \in R_\eps$ for every
$\eps \in \ooi{0, 1}$, so $R_\eps$ is not empty for sufficiently small $\eps > 0$.

The functions $W_{r, \theta, \varphi}$ are usually called \emph{pseudo-critical points} of $I_\eps$ due to the following estimate.
\begin{lem}
\label{pseudo:lem:estimate}
There exists $\eps_0 > 0$ such that
$
\norm*{\nabla I_\eps \parens{W_{r, \theta, \varphi}}}_{H^1}
\lesssim
\eps^2
$
for every $\eps \in \ooi{0, \eps_0}$ and
$
\parens{r, \theta, \varphi}
\in
R_\eps \times \coi{0, 2 \pi} \times \cci{0, \pi}
$.
\end{lem}
\begin{proof}
As $\nabla I_0 \parens{U_{r, \theta, \varphi_j}} = 0$
for every $j \in \set{1, \ldots, K}$, we deduce that
\begin{multline*}
\angles*{\nabla I_\eps \parens{W_{r, \theta, \varphi}} \mid u}_{H^1}
=
\underbrace{
	\int \parens*{\parens{V_\eps - 1} W_{r, \theta, \varphi} u}
}_{\parens{\ast}}
+
\\
+
\eps^3 \int \parens{\phi_{\eps, W_{r, \theta, \varphi}^2} W_{r, \theta, \varphi} u}
-
\underbrace{
	\int \parens*{
		\parens*{
			W_{r, \theta, \varphi}^p
			-
			\sum_{1 \leq j \leq K} U_{r, \theta, \varphi}^p
		}
		u
	}
}_{\parens{\ast \ast}}.
\end{multline*}

\emph{Estimation of $\parens{\ast}$.}
It suffices to estimate
\begin{multline*}
\int \parens*{\parens{V_\eps - 1} U_{r, \theta, \varphi} u}
=
\\
=
\underbrace{
	\int_{D_{\eps, r, \theta, \varphi}}
	\parens*{\parens{V_\eps - 1} U_{r, \theta, \varphi} u}
}_{\parens{\dagger}}
+
\underbrace{
	\int_{B_{\parens{2 \eps}^{- 1}} \parens*{
		P \parens{r, \theta, \varphi}}
	}
	\parens*{\parens{V_\eps - 1} U_{r, \theta, \varphi} u}
}_{\parens{\dagger \dagger}},
\end{multline*}
where
$
D_{\eps, r, \theta, \varphi}
:=
\real^3
\setminus
B_{\parens{2 \eps}^{- 1}} \parens*{P \parens{r, \theta, \varphi}}
$.
Let us estimate $\parens{\dagger}$. In view of \ref{V_3},
\[
\int_{D_{\eps, r, \theta, \varphi}} \parens*{
	\abs{V_\eps - 1} U_{r, \theta, \varphi} \abs{u}
}
\leq
\parens{\sup V + 1}
\parens*{
	\int_{\real^3 \setminus B_{\parens{2 \eps}^{- 1}} \parens*{0}} U^2
}^{1 / 2}
\norm{u}_{L^2}.
\]
Due to \eqref{intro:eqn:exponential_decay}, we have
$
\int_{\real^3 \setminus B_{\parens{2 \eps}^{- 1}} \parens{0}} U^2
\lesssim
e^{- 2 \eta / \eps} / \eps^2
$
for every $\eps \in \ooi{0, \eps_0}$.

Now, we estimate $\parens{\dagger \dagger}$. As
\[
B_{1 / 2} \parens{\eps P \parens{r, \theta, \varphi}}
\subset
\overline{B_1 \parens{0}},
\]
we deduce that
$\abs{V'' \parens{x}} \lesssim 1$
for every
$\eps \in \ooi{0, \eps_0}$,
$
\parens{r, \theta, \varphi}
\in
R_\eps \times \coi{0, 2 \pi} \times \cci{0, \pi}
$
and
$x \in B_{1 / 2} \parens{\eps P \parens{r, \theta, \varphi}}$.
Therefore, it follows from \ref{V_2}, \eqref{intro:eqn:exponential_decay} and a Taylor expansion that
\begin{multline*}
\int_{D_{\eps, r, \theta, \varphi}}
	\parens*{\abs{V_\eps - 1} U_{r, \theta, \varphi} \abs{u}}
\lesssim
\\
\lesssim
\eps
\abs*{\nabla V \parens*{\eps P \parens{r, \theta, \varphi}}}
\int_{B_{\parens{2 \eps}^{- 1}} \parens*{
	P \parens{r, \theta, \varphi}
	}}
	\abs{x - P \parens{r, \theta, \varphi}}
	U_{r, \theta, \varphi} \parens{x}
	\abs{u \parens{x}}
\dif x
+
\\
+
\eps^2 \norm{u}_{L^2}
\end{multline*}
for every
$\eps \in \ooi{0, \eps_0}$,
$
\parens{r, \theta, \varphi}
\in
R_\eps \times \coi{0, 2 \pi} \times \cci{0, \pi}
$
and $u \in H^1$.
We have $3 \parens{\alpha - 1} / \parens{\alpha + 1} > 1$ and \ref{V_2} implies
\[
\abs*{\nabla V \parens*{\eps P \parens{r, \theta, \varphi}}}
\lesssim
\eps^{3 \frac{\alpha - 1}{\alpha + 1}}
\]
for every $\eps \in \ooi{0, \eps_0}$ and
$
\parens{r, \theta, \varphi}
\in
R_\eps \times \coi{0, 2 \pi} \times \cci{0, \pi}
$,
hence the estimate.

\emph{Estimation of $\parens{\ast \ast}$.}
Fix $k > 2 / \abs{d_{2, 1}}$ and given $j \in \set{1, \ldots, K}$, let
\[
\Omega_{\eps, r, \theta, \varphi_j}
=
\set*{
	x \in \real^3:
	k \abs*{x - P \parens{r, \theta, \varphi_j}}
	<
	\frac{\eps^\beta}{\eps^{\frac{\alpha - \lambda}{\alpha + 1}}}
}
\]
and let
$
\Omega_{\eps, 0}
=
\real^3
\setminus
\bigcup_{1 \leq j \leq K} \Omega_{\eps, r, \theta, \varphi_j}
$.
Let us estimate
\[
\parens{\dagger}
:=
\int_{\Omega_{\eps, 0}} \parens*{
	\parens*{
		W_{r, \theta, \varphi}^p
		-
		\sum_{1 \leq j \leq K} U_{r, \theta, \varphi_j}^p
	}
	u
}.
\]
We have
\begin{multline*}
\abs*{\parens{\dagger}}
\leq
\parens{2^p + 1}
\sum_{1 \leq j \leq K}
\int_{\Omega_{\eps, 0}} \parens*{
	U_{r, \theta, \varphi_j}^p u
}
\leq
\\
\leq
\parens{2^p + 1}
\sum_{1 \leq j \leq K}
\parens*{
	\int_{\Omega_{\eps, 0}} U_{r, \theta, \varphi_j}^{2 p}
}^{1 / 2}
\norm{u}_{L^2}.
\end{multline*}
In view of \eqref{intro:eqn:exponential_decay},
\begin{align*}
\sum_{1 \leq j \leq K}
\int_{\Omega_{\eps, 0}} U_{r, \theta, \varphi_j}^{2 p}
&\lesssim
\sum_{1 \leq j \leq K}
\int_{\Omega_{\eps, 0}}
	\exp \parens*{
		- 2 p \eta \abs{x - P \parens{r, \theta, \varphi_j}}
	}
\dif x;
\\
&\lesssim
\exp \parens*{
	- 2 \frac{p}{k} \cdot
	\frac{\eps^\beta}{\eps^{\frac{\alpha - \lambda}{\alpha + 1}}}
}
\end{align*}
for every $\eps \in \ooi{0, \eps_0}$ and
$
\parens{r, \theta, \varphi}
\in
R_\eps \times \coi{0, 2 \pi} \times \cci{0, \pi}
$.
Now, we estimate
\[
\int_{\Omega_{\eps, r, \theta, \varphi_k}} \parens*{
	\parens*{
		W_{r, \theta, \varphi}^p
		-
		\sum_{1 \leq j \leq K} U_{r, \theta, \varphi_j}^p
	}
	u
}.
\]
The function
$f \colon \cci{0, K U \parens{0}} \to \coi{0, \infty}$
given by
$f \parens{t} = t^p$
is Lipschitz, so
\begin{multline*}
\abs*{
	\int_{\Omega_{\eps, r, \theta, \varphi_k}} \parens*{
		\parens*{
			W_{r, \theta, \varphi}^p
			-
			U_{r, \theta, \varphi_k}^p
			-
			\sum_{\substack{1 \leq j \leq K; \\ j \neq k}}
				U_{r, \theta, \varphi_j}^p
		}
		u
	}
}
\leq
\\
\leq
\int_{\Omega_{\eps, r, \theta, \varphi_k}} \parens*{
	\parens*{
		\norm{f}_{C^{0, 1}} \parens*{
			\sum_{\substack{1 \leq j \leq K; \\ j \neq k}}
				U_{r, \theta, \varphi_j}
		}
		+
		\sum_{\substack{1 \leq j \leq K; \\ j \neq k}}
			U_{r, \theta, \varphi_j}^p
	}
	\abs{u}
}
\leq
\\
\leq
2 \norm{f}_{C^{0, 1}}
\sum_{\substack{1 \leq j \leq K; \\ j \neq k}}
\int_{\Omega_{\eps, r, \theta, \varphi_k}} \parens*{
	U_{r, \theta, \varphi_j} \abs{u}
}
\leq
\\
\leq
2 \norm{f}_{C^{0, 1}}
\norm{u}_{L^2}
\sum_{\substack{1 \leq j \leq K; \\ j \neq k}}
	\norm{U_{r, \theta, \varphi_j}}_{
		L^2 \parens{\Omega_{\eps, r, \theta, \varphi_k}}
	}.
\end{multline*}
It follows from \eqref{intro:eqn:exponential_decay} that
\[
\norm{U_{r, \theta, \varphi_j}}_{
	L^2 \parens{\Omega_{\eps, r, \theta, \varphi_k}}
}
\lesssim
\exp \parens*{
	-
	\eta
	\parens*{\abs{d_{j, k}} - \frac{1}{k}}
	\frac{\eps^\beta}
	{\eps^{\parens{\alpha - \lambda} / \parens{\alpha + 1}}}
}
\]
for every $\eps \in \ooi{0, \eps_0}$,
$
\parens{r, \theta, \varphi}
\in
R_\eps \times \coi{0, 2 \pi} \times \cci{0, \pi}
$
and $j \neq k$ in $\set{1, \ldots, K}$, hence the result.
\end{proof}

\subsection{Solving the auxiliary equation}
\label{sect:aux}

In this section, we argue similarly as in \cite[Section 3]{ruizClusterSolutionsSchrodingerPoissonSlater2011} (or \cite[Section 3.3]{ianniConcentrationPositiveBound2008}) to solve the auxiliary equation for sufficiently small $\eps$. More precisely, our present goal is to prove the following result.

\begin{lem}
\label{solving:lem:solution_auxiliary}
We can fix $\eps_0 > 0$ such that if $0 < \eps < \eps_0$, then there exists a mapping of class $C^1$,
\[
R_\eps \times \coi{0, 2 \pi} \times \cci{0, \pi}
\ni \parens{r, \theta, \varphi}
\mapsto
n_{\eps, r, \theta, \varphi} \in H^1,
\]
such that
\[
n_{\eps, r, \theta, \varphi} \in N_{r, \theta, \varphi};
\quad
\Pi_{r, \theta, \varphi} \brackets*{
	\nabla I_\eps \parens{
		W_{r, \theta, \varphi} + n_{\eps, r, \theta, \varphi}
	}
} = 0
\]
and
$\norm{n_{\eps, r, \theta, \varphi}}_{H^1} \lesssim \eps^2$
for every $\eps \in \ooi{0, \eps_0}$ and
$
\parens{r, \theta, \varphi}
\in
R_\eps \times \coi{0, 2 \pi} \times \cci{0, \pi}
$.
\end{lem}

Several preliminaries are needed to prove Lemma \ref{solving:lem:solution_auxiliary}. We begin by letting
$
L_{r, \theta, \varphi}
\colon
N_{r, \theta, \varphi} \to N_{r, \theta, \varphi}
$
be given by
\[
L_{r, \theta, \varphi} \brackets{n}
=
\Pi_{r, \theta, \varphi} \circ \R \brackets{
	I_\eps'' \parens{W_{r, \theta, \varphi}} \brackets{n, \cdot}
}
\]
where $\R \colon H^{- 1} \to H^1$ denotes the Riesz isomorphism.

Our first goal is to obtain sufficient conditions to ``uniformly invert'' $L_{r, \theta, \varphi}$. We need three preliminary lemmas to do so. The first one is the following corollary of the Lax--Milgram Theorem.
\begin{lem}[{\cite[Lemma 3.3]{ianniConcentrationPositiveBound2008}}] \label{lem:LaxMilgram}
Let $H$ be a Hilbert space and let $L \colon H \to H$ be a self-adjoint linear operator. Suppose that $A$ is a finite-dimensional linear subspace of $H$, $c > 0$,
\begin{itemize}
\item
$\angles{L h \mid h}_H \leq -c \norm{h}_H^2$ for every $h \in A$ and
\item
$\angles{L h \mid h}_H \geq c \norm{h}_H^2$ for every $h \in A^\perp$.
\end{itemize}
We conclude that $L$ is invertible and
$\norm{L^{-1}}_{\mathcal{L} \parens{H}} \leq 1/c$.
\end{lem}

The second lemma is a collection of well-known computational results about $I_0'' \parens{U}$.

\begin{lem}[{\cite[Lemma 8.6]{ambrosettiNonlinearAnalysisSemilinear2007}}] \label{auxiliary:lem:ambrosetti_malchiodi}
We have
\begin{enumerate}
\item
$
I_0'' \parens{U} \brackets{U, U} = \parens{1 - p} \norm{U}_{H^1}^2 < 0
$;
\item
$
\ker I_0'' \parens{U}
=
\vspan \set{\partial_i U : i \in \set{1, 2, 3}}
$;
\item
$I_0'' \parens{U} \brackets{w, w} \gtrsim \norm{w}_{H^1}^2$ for every
$w \perp \parens{\ker I_0'' \parens{U} \oplus \vspan \set{U}}$.
\end{enumerate}
\end{lem}

The third lemma collects estimates that follow from the exponential decay \eqref{intro:eqn:exponential_decay}.

\begin{lem}
\label{auxiliary:lem:estimates_from_exponential}
We have
\begin{multline*}
\abs*{
	\angles{
		U_{r, \theta, \varphi_j}
		\mid
		U_{r, \theta, \varphi_k}
	}_{H^1}
},
\,\,\,
\abs*{
	\angles{
		\dot{U}^{\dir_1}_{r, \theta, \varphi_j}
		\mid
		U_{r, \theta, \varphi_k}
	}_{H^1}
},
\\
\abs*{
	\angles{
		\dot{U}^{\dir_1}_{r, \theta, \varphi_j}
		\mid
		\dot{U}^{\dir_2}_{r, \theta, \varphi_k}
	}_{H^1}
}
\lesssim
\exp \parens{- \gamma \sigma \eta r}
\end{multline*}
for every $\dir_1, \dir_2 \in \set{\rad, \azi, \pol}$,
$j \neq k$ in $\set{1, \ldots, K}$ and
$
\parens{r, \theta, \varphi}
\in
\ooi{0, \infty} \times \coi{0, 2 \pi} \times \cci{0, \pi}
$,
where
$\sigma := \min \parens{1, p - 1}$ and $\gamma := \abs{d_{2, 1}}$.
\end{lem}

We proceed to our result of ``uniform inversion'' of
$L_{r, \theta, \varphi}$.

\begin{lem} \label{lem:bound_inverse_L}
There exist $\eps_0, \bar{C} > 0$ such that $L_{r, \theta, \varphi}$ is invertible and
\[
\norm{L_{r, \theta, \varphi}^{-1}}_{
	\End \parens{N_{r, \theta, \varphi}}
}
\leq
\bar{C}
\]
for every $\eps \in \ooi{0, \eps_0}$ and
$
\parens{r, \theta, \varphi}
\in
R_\eps \times \coi{0, 2 \pi} \times \cci{0, \pi}
$.
\end{lem}
\begin{proof}
It suffices to prove that we can apply Lemma \ref{lem:LaxMilgram} in the case $H_{r, \theta, \varphi} := N_{r, \theta, \varphi}$ and
$
A_{r, \theta, \varphi} := \vspan \set{
	\Pi_{r, \theta, \varphi} \brackets{W_{r, \theta, \varphi}}
}
$.

We claim that there exists $\eps_0 > 0$ such that
\[
\angles{L_{r, \theta, \varphi} \brackets{u} \mid u}_{H^1}
\gtrsim
\norm{u}_{H^1}^2
\]
for every $\eps \in \ooi{0, \eps_0}$,
$
\parens{r, \theta, \varphi}
\in
R_\eps \times \coi{0, 2 \pi} \times \cci{0, \pi}
$
and
$u \in A_{r, \theta, \varphi}^\perp$. Indeed, it is easy to check that
\[
\abs*{
	\parens*{
		I_\eps'' \parens{W_{r, \theta, \varphi}}
		-
		I_0'' \parens{W_{r, \theta, \varphi}}
	}
	\brackets{u, u}
}
\lesssim
\eps^3 \norm{u}_{H^1}^2
\]
for every
$\eps > 0$,
$
\parens{r, \theta, \varphi}
\in
\ooi{0, \infty} \times \coi{0, 2 \pi} \times \cci{0, \pi}
$
and $u \in H^1$.
At this point, it suffices to argue precisely as in \cite[p. 262--4]{ruizClusterSolutionsSchrodingerPoissonSlater2011} to conclude.

Now, we claim that there exists $\eps_0 > 0$ such that
\[
- \angles{L_{r, \theta, \varphi} \brackets{u} \mid u}_{H^1}
\lesssim
\norm{u}_{H^1}^2
\]
for every $\eps \in \ooi{0, \eps_0}$,
$
\parens{r, \theta, \varphi}
\in
R_\eps \times \coi{0, 2 \pi} \times \cci{0, \pi}
$
and $u \in A_{r, \theta, \varphi}$. On one hand, Lemma \ref{auxiliary:lem:estimates_from_exponential} shows that
\begin{multline*}
\left|
	I_\eps'' \parens{W_{r, \theta, \varphi}} \brackets*{
		\Pi_{\eps, \beta} \brackets{W_{r, \theta, \varphi}},
		\Pi_{\eps, \beta} \brackets{W_{r, \theta, \varphi}}
	}
	-
\right.
\\
\left.
	-
	I_\eps'' \parens{W_{r, \theta, \varphi}} \brackets{W_{r, \theta, \varphi}, W_{r, \theta, \varphi}}
\right|
\lesssim
e^{- \gamma \sigma \eta r}
\end{multline*}
for every $\eps > 0$ and
$
\parens{r, \theta, \varphi}
\in
R_\eps \times \coi{0, 2 \pi} \times \cci{0, \pi}
$.
On the other hand, it follows from Lemmas \ref{auxiliary:lem:ambrosetti_malchiodi} and \ref{auxiliary:lem:estimates_from_exponential} that, up to shrinking $\eps_0$,
\[
I_\eps'' \parens{W_{r, \theta, \varphi}}
\brackets{W_{r, \theta, \varphi}, W_{r, \theta, \varphi}}
\leq
\frac{K \parens{1 - p}}{2} \norm{U}_{H^1}^2
<
0
\]
for every
$\eps \in \ooi{0, \eps_0}$,
$
\parens{r, \theta, \varphi}
\in
R_\eps \times \coi{0, 2 \pi} \times \cci{0, \pi}
$.
\end{proof}

The next result estimates how $I_\eps''$ varies in a neighborhood of $W_{r, \theta, \varphi}$.

\begin{lem}
\label{auxiliary:lem:estimate_second_derivative}
Given $\eps_0 > 0$, we have
\begin{multline*}
\norm*{
	I_\eps'' \parens{W_{r, \theta, \varphi} + w} \brackets{u, \cdot}
	-
	I_\eps'' \parens{W_{r, \theta, \varphi}} \brackets{u, \cdot}
}_{H^{-1}}
\lesssim
\\
\lesssim
\parens*{
	\norm{w}_{H^1} + \norm{w}_{H^1}^2 + \norm{w}^{p-1}_{H^1}
}
\norm{u}_{H^1}
\end{multline*}
for every
$\eps \in \ooi{0, \eps_0}$;
$
\parens{r, \theta, \varphi}
\in
\ooi{0, \infty} \times \coi{0, 2 \pi} \times \cci{0, \pi}
$
and $u, w \in H^1$.
\end{lem}
\begin{proof}
Clearly,
\begin{multline*}
\parens*{I_\eps'' \parens{W_{r, \theta, \varphi} + w} - I_\eps'' \parens{W_{r, \theta, \varphi}}}
\brackets{u_1, u_2}
=
\underbrace{
	\eps^3
	\int \parens{\phi_{\eps, w^2} u_1 u_2 + \phi_{\eps, w u_1} w u_2}
}_{\parens{\ast}}
+
\\
+
\underbrace{
	\eps^3
	\int \parens{
		\phi_{\eps, W_{r, \theta, \varphi} u_1} w u_2
		+
		\phi_{\eps, w u_1} W_{r, \theta, \varphi} u_2
		+
		2 \phi_{\eps, W_{r, \theta, \varphi} w} u_1 u_2
	}
}_{\parens{\ast \ast}}
-
\\
-
\underbrace{
	p \int \parens*{
		\parens*{\abs{W_{r, \theta, \varphi} + w}^{p - 1} - W_{r, \theta, \varphi}^{p - 1}} u_1 u_2
	}
}_{\parens{\ast \ast \ast}}.
\end{multline*}
We can estimate the terms on the right-hand side as follows:
\[
\abs{\parens{\ast}}
\leq
\eps^3 \norm{w}_{H^1}^2 \norm{u_1}_{H^1} \norm{u_2}_{H^1};
\quad \quad
\parens{\ast \ast}
\leq
\eps^3 \norm{w}_{H^1} \norm{u_1}_{H^1} \norm{u_2}_{H^1}
\]
and
\[
\parens{\ast \ast \ast}
\lesssim
\parens*{\norm{w}_{H^1} + \norm{w}_{H^1}^{p - 1}}
\norm{u_1}_{H^1} \norm{u_2}_{H^1}
\]
for every $\eps, r > 0$ and $w, u_1, u_2 \in H^1$.
\end{proof}

The following lemma is the last preliminary result.

\begin{lem}
\label{auxiliary:lem:Banach}
We can fix $\eps_0, \bar{C} > 0$ such that if $0 < \eps < \eps_0$ and 
$
\parens{r, \theta, \varphi}
\in
R_\eps \times \coi{0, 2 \pi} \times \cci{0, \pi}
$,
then the problem
\begin{equation}
\label{auxiliary:eqn:Banach}
\Pi_{r, \theta, \varphi} \brackets{
	\nabla I_\eps \parens{W_{r, \theta, \varphi} + n}
} = 0;
\quad \quad
n \in \mathcal{N}_{\eps, r, \theta, \varphi}^{\bar{C}}
\end{equation}
has a unique solution, where
\[
\mathcal{N}_{\eps, r, \theta, \varphi}^{\bar{C}}
:=
\set*{
	n \in N_{r, \theta, \varphi}:
	\norm{n}_{H^1}
	\leq
	2 \bar{C} \norm{\nabla I_\eps \parens{W_{r, \theta, \varphi}}}_{H^1}
}.
\]
\end{lem}
\begin{proof}
Let $\eps_0, \bar{C} > 0$ be furnished by Lemma \ref{lem:bound_inverse_L} and let
$
S_{\eps, r, \theta, \varphi}
\colon
N_{r, \theta, \varphi} \to N_{r, \theta, \varphi}
$
be given by
\[
S_{\eps, r, \theta, \varphi} \parens{n}
=
n
-
L_{r, \theta, \varphi}^{-1} \circ \Pi_{r, \theta, \varphi}
\brackets*{\nabla I_\eps \parens{W_{r, \theta, \varphi} + n}}
\]
for any $\eps \in \ooi{0, \eps_0}$ and
$
\parens{r, \theta, \varphi}
\in
R_\eps \times \coi{0, 2 \pi} \times \cci{0, \pi}
$.
To finish from an application of the Banach Fixed Point Theorem, we only have to prove that there exists $\eps_1 > 0$ such that given
$0 < \eps < \eps_1$ and
$
\parens{r, \theta, \varphi}
\in
R_\eps \times \coi{0, 2 \pi} \times \cci{0, \pi}
$,
$S_{\eps, r, \theta, \varphi}$ has a unique fixed point in
$\mathcal{N}_{\eps, r, \theta, \varphi}^{\bar{C}}$.

\begin{claim}
There exists $\eps_1 > 0$ such that
\[
\norm*{S_{\eps, r, \theta, \varphi}' \parens{n}}_{
	\End \parens{N_{r, \theta, \varphi}}
} 
\leq
\frac{1}{2}
\]
for every $\eps \in \ooi{0, \eps_1}$,
$
\parens{r, \theta, \varphi}
\in
R_\eps \times \coi{0, 2 \pi} \times \cci{0, \pi}
$
and $n \in \mathcal{N}_{\eps, r, \theta, \varphi}^{\bar{C}}$.
\end{claim}
\begin{proof}
Indeed,
\begin{align*}
S_{\eps, r, \theta, \varphi}' \parens{w_1} \brackets{w_2}
&= w_2 - L_{r, \theta, \varphi}^{-1} \circ \R \circ I_\eps'' \parens{W_{r, \theta, \varphi} + w_1}
\brackets{w_2, \cdot};
\\
&=
L_{r, \theta, \varphi}^{-1}
\brackets*{
	L_{r, \theta, \varphi} \brackets{w_2}
	-
	\R \circ I_\eps'' \parens{W_{r, \theta, \varphi} + w_1} \brackets{w_2, \cdot}
};
\\
&=
L_{r, \theta, \varphi}^{-1} \circ \R \brackets*{
	I_\eps'' \parens{W_{r, \theta, \varphi}} \brackets{w_2, \cdot}
	-
	I_\eps'' \parens{W_{r, \theta, \varphi} + w_1} \brackets{w_2, \cdot}
}.
\end{align*}
Due to Lemma \ref{auxiliary:lem:estimate_second_derivative},
\begin{multline*}
\abs*{S_{\eps, r, \theta, \varphi}' \parens{w_1} \brackets{w_2}}
\leq
\\
\leq
\abs*{
	L_{r, \theta, \varphi}^{-1} \circ \R
	\brackets*{
		I_\eps'' \parens{W_{r, \theta, \varphi}} \brackets{w_2, \cdot}
		-
		I_\eps'' \parens{W_{r, \theta, \varphi} + w_1} \brackets{w_2, \cdot}
	}
}
\lesssim
\\
\lesssim
\parens*{
	\norm{w_1}_{H^1} + \norm{w_1}_{H^1}^2 + \norm{w_1}_{H^1}^{p - 1}
}
\norm{w_2}_{H^1}
\end{multline*}
for every $\eps \in \ooi{0, \eps_0}$; $r \in R_\eps$
and $w_1, w_2 \in N_{r, \theta, \varphi}$, so the result follows from Lemma \ref{pseudo:lem:estimate}.
\end{proof}

At this point, we only have to prove that
$
S_{\eps, r, \theta, \varphi}
\parens{\mathcal{N}_{\eps, r, \theta, \varphi}^{\bar{C}}}
\subset
\mathcal{N}_{\eps, r, \theta, \varphi}^{\bar{C}}
$
for every $\eps \in \ooi{0, \eps_2}$ and
$
\parens{r, \theta, \varphi}
\in
R_\eps \times \coi{0, 2 \pi} \times \cci{0, \pi}
$,
where
$\eps_2 := \min \parens{\eps_0, \eps_1} > 0$. Due to Lemma \ref{lem:bound_inverse_L},
\[
\norm*{S_{\eps, r, \theta, \varphi} \parens{0}}_{H^1}
=
\norm*{
	L_{r, \theta, \varphi}^{-1} \brackets{\nabla I_\eps \parens{W_{r, \theta, \varphi}}}
}_{H^1}
\leq
\bar{C} \norm{\nabla I_\eps \parens{W_{r, \theta, \varphi}}}_{H^1}.
\]
In view of the previous claim,
\[
\norm{S_{\eps, r, \theta, \varphi} \parens{0} - S_{\eps, r, \theta, \varphi} \parens{w}}_{H^1}
\leq
\frac{1}{2}\norm{w}_{H^1}
\leq
\bar{C} \norm{\nabla I_\eps \parens{W_{r, \theta, \varphi}}}_{H^1}.
\]
We conclude that
$
\norm{S_{\eps, r, \theta, \varphi} \parens{w}}_{H^1}
\leq
2 \bar{C} \norm{\nabla I_\eps \parens{W_{r, \theta, \varphi}}}_{H^1}
$,
hence the result.
\end{proof}

We can finally prove Lemma \ref{solving:lem:solution_auxiliary}.

\begin{proof}[Proof of Lemma \ref{solving:lem:solution_auxiliary}]
Fix $\eps_0 > 0$ for which the conclusions of Lemmas \ref{lem:bound_inverse_L} and \ref{auxiliary:lem:Banach} hold and let $\bar{C} > 0 $ be furnished by Lemma \ref{auxiliary:lem:Banach}. Let
$
\mathcal{H}
\colon
\mathcal{O} \to H^1
$
be the mapping of class $C^1$ given by
\[
\mathcal{H} \parens{\eps, r, \theta, \varphi, n}
=
\nabla I_\eps \parens{W_{r, \theta, \varphi} + n},
\]
where
\begin{multline*}
\mathcal{O}
:=
\left\{
	\parens{\eps, r, \theta, \varphi, n}
	\in
	\ooi{0, \eps_0}
	\times
	\ooi{0, \infty} \times \coi{0, 2 \pi} \times \cci{0, \pi}
	\times
	H^1:
\right.
\\
\left.
	r \in R_\eps
	\quad \text{and} \quad
	n \in \mathcal{N}_{\eps, r, \theta, \varphi}^{\bar{C}}
\right\}.
\end{multline*}

The next claim follows directly from Lemmas \ref{lem:bound_inverse_L} and \ref{auxiliary:lem:estimate_second_derivative}.

\begin{claim} \label{claim:x}
Up to shrinking $\eps_0$,
\[
N_{r, \theta, \varphi} \ni n_2
\mapsto
\mathcal{H}'_{\eps, r, \theta, \varphi} \parens{n_1} \brackets{n_2}
:=
\mathcal{H}' \parens{\eps, r, \theta, \varphi, n_1}
\brackets{0, 0, 0, 0, n_2}
\in
H^1
\]
is invertible and
\[
\norm*{
	\mathcal{H}'_{\eps, r, \theta, \varphi} \parens{n} \brackets{u}
}_{H^1}
\gtrsim
\norm{u}_{H^1}
\]
for every $\parens{\eps, r, \theta, \varphi, n} \in \mathcal{O}$
and $u \in H^1$.
\end{claim}

In view of the previous claim, we can use the Implicit Function Theorem to deduce that there exists a mapping of class $C^1$,
\[
R_\eps \times \coi{0, 2 \pi} \times \cci{0, \pi}
\ni
\parens{r, \theta, \varphi}
\mapsto
n_{\eps, r, \theta, \varphi}
\in
H^1,
\]
such that
\[
\mathcal{H}
\parens{\eps, r, \theta, \varphi, n_{\eps, r, \theta, \varphi}}
=
0
\]
for every $\eps \in \ooi{0, \eps_0}$ and
$
\parens{r, \theta, \varphi}
\in
R_\eps \times \coi{0, 2 \pi} \times \cci{0, \pi}
$.
To finish, the estimate follows from Lemma \ref{pseudo:lem:estimate}.
\end{proof}

\subsection{The reduced functional}

In this section (based on \cite[Section 4]{ruizClusterSolutionsSchrodingerPoissonSlater2011}), we define the reduced functional $\Phi_\eps$, we show that its critical points are naturally associated with critical points of $I_\eps$ and we prove that $\Phi_\eps$ has a critical point.

Let us begin by defining the reduced functional.

\begin{defn}
Suppose that $\eps_0 > 0$ is furnished by Lemma \ref{solving:lem:solution_auxiliary}. If $0 < \eps < \eps_0$, then we define the \emph{reduced functional}
$
\Phi_\eps
\in
C^1 \parens{R_\eps \times \coi{0, 2 \pi} \times \cci{0, \pi}}
$
as given by
\[
\Phi_\eps \parens{r, \theta, \varphi}
=
I_\eps \parens{W_{r, \theta, \varphi} + n_{\eps, r, \theta, \varphi}}.
\]
\end{defn}

The next result shows that if $\eps$ is sufficiently small, then critical points of $\Phi_\eps$ are naturally associated with critical points of $I_\eps$.

\begin{lem} \label{lem:natural_constraint}
There exists $\eps_0 > 0$ such that if $0 < \eps < \eps_0$ and
$\nabla \Phi_\eps \parens{r, \theta, \varphi} = 0$, then
$
\nabla I_\eps \parens {W_{r, \theta, \varphi}
+
n_{\eps, r, \theta, \varphi}}
=
0
$.
\end{lem}
\begin{proof}
Suppose that $\eps_0 > 0$ is furnished by Lemma \ref{solving:lem:solution_auxiliary}, $0 < \eps < \eps_0$ and
$
\parens{r, \theta, \varphi}
\in
R_\eps \times \coi{0, 2 \pi} \times \cci{0, \pi}
$.
Due to Lemma \ref{solving:lem:solution_auxiliary},
\[
\nabla I_\eps \parens{
	W_{r, \theta, \varphi} + n_{\eps, r, \theta, \varphi}
}
=
\sum_{\dir \in \set{\rad, \azi, \pol}}
c_{\eps, r, \theta, \varphi}^{\dir}
\dot{W}_{r, \theta, \varphi}^{\dir}
\]
for a certain
$
\parens{
	c_{\eps, r, \theta, \varphi}^{\rad}; ~
	c_{\eps, r, \theta, \varphi}^{\azi}; ~
	c_{\eps, r, \theta, \varphi}^{\pol}
}
\in \real^3
$.
It then follows from the definition of $\Phi_\eps$ that
\begin{align*}
\partial_{\dir_1} \Phi_\eps \parens{r, \theta, \varphi}
&=
\angles*{
	\dot{W}_{r, \theta, \varphi}^{\dir_1}
	+
	\dot{n}_{\eps, r, \theta, \varphi}^{\dir_1}
	~\middle|~
	\nabla I_\eps \parens{
		W_{r, \theta, \varphi} + n_{\eps, r, \theta, \varphi}
	}
}_{H^1};
\\
&=
\sum_{\dir_2 \in \set{\rad, \azi, \pol}}
c_{\eps, r, \theta, \varphi}^{\dir_2}
\angles*{
	\dot{W}_{r, \theta, \varphi}^{\dir_1}
	+
	\dot{n}_{\eps, r, \theta, \varphi}^{\dir_1}
	~ \middle| ~
	\dot{W}_{r, \theta, \varphi}^{\dir_2}
}_{H^1}.
\end{align*}
At this point, the result is a consequence of the following claim.
\begin{claim}
\label{reduced:claim:identity_matrix}
There exists $\eps_1 \in \ooi{0, \eps_0}$ such that
\[
\abs*{
	\angles*{
		\dot{W}_{r, \theta, \varphi}^{\dir_1}
		+
		\dot{n}_{\eps, r, \theta, \varphi}^{\dir_1}
		~ \middle| ~
		\dot{W}_{r, \theta, \varphi}^{\dir_2}
	}_{H^1}
	-
	K \norm{\partial_1 U}_{H^1}^2 \delta^{\dir_1, \dir_2}
}
\lesssim
\eps^2
\]
for every $\eps \in \ooi{0, \eps_1}$ and
$
\parens{r, \theta, \varphi}
\in
R_\eps \times \coi{0, 2 \pi} \times \cci{0, \pi}
$.
\end{claim}
\begin{proof}
Clearly,
\[
\angles*{
	\dot{W}_{r, \theta, \varphi}^{\dir_1}
	+
	\dot{n}_{\eps, r, \theta, \varphi}^{\dir_1}
	~ \middle| ~
	\dot{W}_{r, \theta, \varphi}^{\dir_2}
}_{H^1}
=
\underbrace{
	\angles{
		\dot{n}_{\eps, r, \theta, \varphi}^{\dir_1}
		\mid
		\dot{W}_{r, \theta, \varphi}^{\dir_2}
	}_{H^1}
}_{\parens{\ast}}
+
\underbrace{
	\angles{
		\dot{W}_{r, \theta, \varphi}^{\dir_1}
		\mid
		\dot{W}_{r, \theta, \varphi}^{\dir_2}
	}_{H^1}
}_{\parens{\ast \ast}}.
\]
Let us estimate $\parens{\ast}$. As
$
\angles{
	n_{\eps, r, \theta, \varphi}
	\mid
	\dot{W}_{r, \theta, \varphi}^{\dir_2}
}_{H^1} = 0
$
for every
$
\parens{r, \theta, \varphi}
\in
R_\eps \times \coi{0, 2 \pi} \times \cci{0, \pi}
$,
we can differentiate this expression to deduce that
\[
\abs{\parens{\ast}}
\leq
\norm{n_{\eps, r, \theta, \varphi}}_{H^1}
\norm{\ddot{W}_{r, \theta, \varphi}^{\dir_1, \dir_2}}_{H^1}.
\]
In view of Lemma \ref{solving:lem:solution_auxiliary},
we have $\abs{\parens{\ast}} \lesssim \eps^2$ for every
$\eps \in \ooi{0, \eps_0}$
and
$
\parens{r, \theta, \varphi}
\in
R_\eps \times \coi{0, 2 \pi} \times \cci{0, \pi}
$.

Now, consider $\parens{\ast \ast}$. We begin by treating the diagonal terms, i.e., when $\dir := \dir_1 = \dir_2$. It follows from Lemma \ref{auxiliary:lem:estimates_from_exponential} that
\[
\abs*{
	\norm{\dot{W}^{\dir}_{r, \theta, \varphi}}_{H^1}^2
	-
	K \norm{\partial_1 U}_{H^1}^2
}
\lesssim
\exp \parens{- \gamma \sigma \eta r}
\]
for every $\dir \in \set{\rad, \azi, \pol}$ and
$
\parens{r, \theta, \varphi}
\in
\ooi{0, \infty} \times \coi{0, 2 \pi} \times \cci{0, \pi}
$.
We proceed to the off-diagonal terms. In view of Lemma \ref{auxiliary:lem:estimates_from_exponential}, we have
$
\abs{\parens{\ast \ast}}
\lesssim
\exp \parens{- \gamma \sigma \eta r}
$
for every $\dir_1 \neq \dir_2$ in $\set{\rad, \azi, \pol}$ and
$
\parens{r, \theta, \varphi}
\in
\ooi{0, \infty} \times \coi{0, 2 \pi} \times \cci{0, \pi}
$.
\end{proof}
\end{proof}

Our next result provides an approximation for the function
$\phi_{\eps, U_{r, \theta, \varphi}^2}$.

\begin{lem}
\label{reduced:lem:approximation_for_phi_eps}
We have
\[
\abs*{
	\phi_{\eps, U_{r, \theta, \varphi}^2} \parens{x}
	-
	\kappa_\eps \parens{x - P \parens{r, \theta, \varphi}}
	\norm{U}_{L^2}^2
}
\lesssim
\eps
\]
for every $\eps > 0$, $x \in \real^3$ and
$
\parens{r, \theta, \varphi}
\in
\ooi{0, \infty} \times \coi{0, 2 \pi} \times \cci{0, \pi}
$.
\end{lem}
\begin{proof}
A change of variable shows that
\begin{multline*}
\abs*{
	\phi_{\eps, U_{r, \theta, \varphi}^2} \parens{x}
	-
	\kappa_\eps \parens*{x - P \parens{r, \theta, \varphi}}
	\norm{U}_{L^2}^2
}
=
\\
=
\abs*{
	\int
		\kappa_\eps \parens*{
			x-y-P \parens{r, \theta, \varphi}
		}
		U \parens{y}^2
	\dif y
	-
	\kappa_\eps \parens*{x - P \parens{r, \theta, \varphi}}
	\norm{U}_{L^2}^2
}.
\end{multline*}
As $\sup \abs{\nabla \kappa} = 1 / 2$, we obtain
\[
\abs*{
	\phi_{\eps, U_{r, \theta, \varphi}^2} \parens{x}
	-
	\kappa_\eps \parens*{x - P \parens{r, \theta, \varphi}}
	\norm{U}_{L^2}^2
}
\leq
\frac{\eps}{2} \int \abs{y} U \parens{y}^2 \dif y,
\]
hence the result.
\end{proof}

Let us obtain an expansion for
$I_\eps \parens{W_{r, \theta, \varphi}}$.

\begin{lem}
\label{reduced:lem:expansion_of_Ieps}
There exist $\eps_0 > 0$ and
$\mu \in \oci{\parens{\lambda + 1} / \parens{\alpha + 1}, 1}$
such that
\begin{multline*}
\left|
	I_\eps \parens{W_{r, \theta, \varphi}}
	-
	K C_0
	-
	C_1 \parens*{
		\sum_{1 \leq j \leq K}
		V_\eps \parens*{P \parens{r, \theta, \varphi_j}}
	}
	-
\right.
\\
\left.
	-
	C_1^2 \eps^3
	\parens*{
		K
		+
		2
		\sum_{1 \leq j < k \leq K}
			\kappa_\eps \parens{r d_{j, k}}
	}
\right|
\lesssim
\eps^{3 + \mu}
\end{multline*}
for every
$\eps \in \ooi{0, \eps_0}$ and
$
\parens{r, \theta, \varphi}
\in
\overline{R_\eps} \times \coi{0, 2 \pi} \times \cci{0, \pi}
$,
where
\[
C_0
:=
\frac{1}{2} \norm{U}_{D^{1, 2}}^2
-
\frac{1}{p + 1}\norm{U}_{L^{p + 1}}^{p + 1}
\quad \text{and} \quad
C_1 := \frac{1}{2} \norm{U}_{L^2}^2.
\]
\end{lem}
\begin{proof}
As
\[
\norm{W_{r, \theta, \varphi}}_{H^1_\eps}^2
=
\parens*{
	\sum_{1 \leq j \leq K}
	\norm*{
		U_{r, \theta, \varphi_j}
	}_{H^1_\eps}^2
}
+
2
\sum_{1 \leq j < k \leq K}
\angles{
	U_{r, \theta, \varphi_j} \mid U_{r, \theta, \varphi_k}
}_{H^1_\eps},
\]
\[
I_\eps\parens{U_{r, \theta, \varphi_j}}
=
C_0
+
\frac{1}{2} \int \parens*{V_\eps U_{r, \theta, \varphi_j}^2}
+
\frac{\eps^3}{4}
\int \parens*{
	\phi_{\eps, U_{r, \theta, \varphi_j}^2}
	U_{r, \theta, \varphi_j}^2
}
\]
and
\begin{multline*}
\int \parens*{
	\phi_{\eps, W_{r, \theta, \varphi}^2} W_{r, \theta, \varphi}^2
}
=
\parens*{
	\sum_{1 \leq j \leq K}
	\int \parens*{
		\phi_{\eps, U_{r, \theta, \varphi_j}^2}
		U_{r, \theta, \varphi_j}^2
	}
}
+
\\
+
2
\parens*{
	\sum_{1 \leq j < k \leq K}
	\int \parens*{
		\phi_{\eps, U_{r, \theta, \varphi_j}^2}
		U_{r, \theta, \varphi_k}^2
	}
}
+
\\
+
2 \parens*{
	\sum_{\substack{1 \leq j \leq K; \\ 1 \leq k < l \leq K}}
	\int \parens*{
		\phi_{\eps, U_{r, \theta, \varphi_j}^2}
		U_{r, \theta, \varphi_k}
		U_{r, \theta, \varphi_l}
	}
}
+
\\
+
2
\sum_{1 \leq j < k \leq K}
\int \parens*{
	\phi_{\eps, W_{r, \theta, \varphi}^2}
	U_{r, \theta, \varphi_j}
	U_{r, \theta, \varphi_k}
},
\end{multline*}
we deduce that
\begin{multline*}
I_\eps \parens{W_{r, \theta, \varphi}}
=
K C_0
+
\frac{1}{2}
\parens*{
	\sum_{1 \leq j \leq K}
	\int \parens*{
		V_\eps U_{r, \theta, \varphi_j}^2
	}
}
+
\\
+
\frac{\eps^3}{4}
\parens*{
	\sum_{1 \leq j \leq K}
	\int \parens*{
		\phi_{\eps, U_{r, \theta, \varphi_j}^2}
		U_{r, \theta, \varphi_j}^2
	}
}
+
\parens*{
	\sum_{1 \leq j < k \leq K}
	\angles{
		U_{r, \theta, \varphi_j} \mid U_{r, \theta, \varphi_k}
	}_{H^1_\eps}
}
+
\\
+
\frac{\eps^3}{2}
\parens*{
	\sum_{1 \leq j < k \leq K}
	\int \parens*{
		\phi_{\eps, U_{r, \theta, \varphi_j}^2}
		U_{r, \theta, \varphi_k}^2
	}
}
+
\\
+
\frac{\eps^3}{2}
\parens*{
	\sum_{\substack{1 \leq j \leq K; \\ 1 \leq k < l \leq K}}
	\int \parens*{
		\phi_{\eps, U_{r, \theta, \varphi_j}^2}
		U_{r, \theta, \varphi_k}
		U_{r, \theta, \varphi_l}
	}
}
+
\\
+
\frac{\eps^3}{2}
\parens*{
	\sum_{1 \leq j < k \leq K}
	\int \parens*{
		\phi_{\eps, W_{r, \theta, \varphi}^2}
		U_{r, \theta, \varphi_j}
		U_{r, \theta, \varphi_k}
	}
}
-
\\
-
\frac{1}{p+1}
\parens*{
	\norm{W_{r, \theta, \varphi}}_{L^{p+1}}^{p+1}
	-
	K \norm{U}_{L^{p + 1}}^{p + 1}
}.
\end{multline*}

\emph{Expansion of
$\parens{1 / 2} \int \parens{V_\eps U_{r, \theta, \varphi_j}^2}$.}
Up to shrinking $\eps_0$, we can suppose that
$\eps_0 r + \eps_0^{1 - \beta} < 2$
for every $r \in \overline{R_{\eps_0}}$.
Consider the decomposition
\[
\real^3
=
\parens*{
	\real^3
	\setminus
	B_{\eps^{- \beta}} \parens*{P \parens{r, \theta, \varphi_j}}
}
\cup
B_{\eps^{- \beta}} \parens*{P \parens{r, \theta, \varphi_j}}.
\]
An estimate for the integral over
$
\real^3
\setminus
B_{\eps^{- \beta}} \parens*{P \parens{r, \theta, \varphi_j}}
$
follows directly from the exponential decay \eqref{intro:eqn:exponential_decay}. A Taylor expansion shows that
\begin{multline*}
\left|
\int_{
	B_{\eps^{- \beta}} \parens*{P \parens{r, \theta, \varphi_j}}
}
	U_{r, \theta, \varphi_j} \parens{x}^2
	\left(
		V_\eps \parens{x}
		-
		V_\eps \parens*{P \parens{r, \theta, \varphi_j}}
		-
	\right.
\right.
\\
\left.
	\left.
		-
		\eps
		\nabla V \parens{\eps P \parens{r, \theta, \varphi_j}}
		\cdot
		\parens*{x - P \parens{r, \theta, \varphi_j}}
	\right)
\dif x
\right|
\leq
\\
\leq
\underbrace{
	\frac{\eps^2}{2}
	\abs*{V'' \parens*{\eps P \parens{r, \theta, \varphi_j}}}
	\int
		\abs{x - P \parens{r, \theta, \varphi_j}}^2
		U_{r, \theta, \varphi_j} \parens{x}^2
	\dif x
}_{\parens{\ast}}
+
\\
+
\underbrace{
	\frac{\eps^3}{3!}
	\parens*{
		\sup_{
			x
			\in
			B_{\eps^{1 - \beta}}
			\parens{P \parens{r, \theta, \varphi_j}}
		}
		\abs*{V''' \parens{x}}
	}
	\int
		\abs{x - P \parens{r, \theta, \varphi_j}}^3
		U_{r, \theta, \varphi_j} \parens{x}^2
	\dif x
}_{\parens{\ast \ast}}.
\end{multline*}
As $U_{r, \theta, \varphi_j}$ is spherically symmetric around
$P \parens{r, \theta, \varphi_j}$, we deduce that
\[
\int_{
	B_{\eps^{- \beta}} \parens*{P \parens{r, \theta, \varphi_j}}
}
	U_{r, \theta, \varphi_j} \parens{x}^2
	\nabla V \parens{\eps P \parens{r, \theta, \varphi_j}}
	\cdot
	\parens*{x - P \parens{r, \theta, \varphi_j}}
\dif x
=
0.
\]

Let us estimate $\parens{\ast}$. It follows from \ref{V_2} and the definition of $R_\eps$ that, up to shrinking $\eps_0$,
\[
\abs*{V'' \parens*{\eps P \parens{r, \theta, \varphi_j}}}
\lesssim
\eps^{3 \frac{\alpha - 2}{\alpha + 1}}
\]
for every
$\eps \in \ooi{0, \eps_0}$ and
$
\parens{r, \theta, \varphi}
\in
\overline{R_\eps} \times \coi{0, 2 \pi} \times \cci{0, \pi}
$.
We just obtained a good enough estimate for the purposes of the lemma. Indeed,
\[
2 + 3 \frac{\alpha - 2}{\alpha + 1}
>
3 + \frac{\lambda + 1}{\alpha + 1}
\]
because $\lambda < 2 \alpha - 8$.

Now, we estimate $\parens{\ast \ast}$. On one hand, it follows from \ref{V_2} that
\[
\abs*{V''' \parens*{\eps P \parens{r, \theta, \varphi_j}}}
\lesssim
\eps^{3 \frac{\alpha - 3}{\alpha + 1}}
\]
for every $\eps \in \ooi{0, \eps_0}$ and
$
\parens{r, \theta, \varphi}
\in
\overline{R_\eps} \times \coi{0, 2 \pi} \times \cci{0, \pi}
$.
On the other hand,
\[
\abs*{
	V''' \parens{x}
	-
	V''' \parens*{\eps P \parens{r, \theta, \varphi_j}}
}
<
\eps^{1 - \beta}
\]
for every
$x \in B_{\eps^{1 - \beta}} \parens{P \parens{r, \theta, \varphi_j}}$. Therefore,
\[
\sup_{
	x
	\in
	B_{\eps^{1 - \beta}} \parens{P \parens{r, \theta, \varphi_j}}
}
\abs*{V''' \parens{x}}
\lesssim
\eps^{\min \parens*{3 \frac{\alpha - 3}{\alpha + 1}, 1 - \beta}}
\]
for every $\eps \in \ooi{0, \eps_0}$ and
$
\parens{r, \theta, \varphi}
\in
\overline{R_\eps} \times \coi{0, 2 \pi} \times \cci{0, \pi}
$.
We claim that this estimate suffices for the purposes of the lemma. Indeed,
\[
\min \parens*{3 \frac{\alpha - 3}{\alpha + 1}, 1 - \beta}
>
\frac{\lambda + 1}{\alpha + 1}
\]
because
$\beta < \parens{\alpha - \lambda} / \parens{\alpha + 1}$
and
$\lambda < 3 \alpha - 10$.

\emph{Expansion of
$
\parens{\eps^3 / 4}
\int \parens{
	\phi_{\eps, U_{r, \theta, \varphi}^2}
	U_{r, \theta, \varphi}^2
}
$
.}
In view of \eqref{intro:eqn:exponential_decay} and Lemma \ref{reduced:lem:approximation_for_phi_eps},
\begin{multline*}
\abs*{
	\int \parens*{
		\phi_{\eps, U_{r, \theta, \varphi}^2}
		U_{r, \theta, \varphi}^2
	}
	-
	\phi_{\eps, U_{r, \theta, \varphi}^2}
	\parens*{P \parens{r, \theta, \varphi}}
	\norm{U}_{L^2}^2
}
=
\\
=
\abs*{
	\int \parens*{
		\phi_{\eps, U_{r, \theta, \varphi}^2}
		U_{r, \theta, \varphi}^2
	}
	-
	\norm{U}_{L^2}^2
	\int
		\kappa_\eps \parens*{x - P \parens{r, \theta, \varphi}}
		U_{r, \theta, \varphi} \parens{x}^2
	\dif x
}
\lesssim
\eps
\end{multline*}
and
$
\abs{
	\phi_{\eps, U_{r, \theta, \varphi}^2}
	\parens{P \parens{r, \theta, \varphi}}
	-
	\norm{U}_{L^2}^2
}
\lesssim
\eps
$,
so
\[
\abs*{
	\frac{\eps^3}{4}
	\int \parens*{
		\phi_{\eps, U_{r, \theta, \varphi}^2}
		U_{r, \theta, \varphi}^2
	}
	-
	\frac{\eps^3}{4} \norm{U}_{L^2}^4
}
\lesssim
\eps^4
\]
for every $\eps > 0$ and
$
\parens{r, \theta, \varphi}
\in
\ooi{0, \infty} \times \coi{0, 2 \pi} \times \cci{0, \pi}
$.

\emph{Estimation of
$
\angles{
	U_{r, \theta, \varphi_j} \mid U_{r, \theta, \varphi_k}
}_{H^1_\eps}
$
for $j \neq k$.}
Clearly,
\[
\angles{
	U_{r, \theta, \varphi_j} \mid U_{r, \theta, \varphi_k}
}_{H^1_\eps}
=
\angles{
	U_{r, \theta, \varphi_j} \mid U_{r, \theta, \varphi_k}
}_{H^1}
+
\int \parens*{
	\parens{V_\eps - 1}
	U_{r, \theta, \varphi_j}
	U_{r, \theta, \varphi_k}
}.
\]
As $U_{r, \theta, \varphi_j}$ solves \eqref{intro:eqn:NLSE}, we obtain
\[
\angles{
	U_{r, \theta, \varphi_j} \mid U_{r, \theta, \varphi_k}
}_{H^1_\eps}
=
\int \parens*{U_{r, \theta, \varphi_j}^p U_{r, \theta, \varphi_k}}
+
\int \parens*{
	\parens{V_\eps - 1}
	U_{r, \theta, \varphi_j}
	U_{r, \theta, \varphi_k}
}.
\]
At this point, the estimate follows from the exponential decay \eqref{intro:eqn:exponential_decay}.

\emph{Expansion of
$
\int \parens{
	\phi_{\eps, U_{r, \theta, \varphi_j}^2}
	U_{r, \theta, \varphi_k}^2
}
$
for $j \neq k$.
}
It follows from Lemma \ref{reduced:lem:approximation_for_phi_eps} that
\[
\abs*{
	\frac{\eps^3}{2}
	\int \parens*{
		\phi_{\eps, U_{r, \theta, \varphi_j}^2}
		U_{r, \theta, \varphi_k}^2
	}
	-
	\frac{\eps^3}{2}
	\norm{U}_{L^2}^4
	\kappa_\eps \parens{r d_{j, k}}
}
\lesssim
\eps^4
\]
for every $\eps > 0$ and
$
\parens{r, \theta, \varphi}
\in
\ooi{0, \infty} \times \coi{0, 2 \pi} \times \cci{0, \pi}
$.

\emph{Estimation of
\[
\int \parens*{
	\phi_{\eps, W_{r, \theta, \varphi}^2}
	U_{r, \theta, \varphi_j}
	U_{r, \theta, \varphi_k}
}
\quad \text{and} \quad
\int \parens*{
	\phi_{\eps, U_{r, \theta, \varphi_l}^2}
	U_{r, \theta, \varphi_j}
	U_{r, \theta, \varphi_k}
}
\]
for $j \neq k$.
}
Corollary of the exponential decay \eqref{intro:eqn:exponential_decay}.

\emph{Estimation of
$
\norm{W_{r, \theta, \varphi}}_{L^{p+1}}^{p+1}
-
K \norm{U}_{L^{p + 1}}^{p + 1}
$.}
It suffices to argue as in the proof of Lemma \ref{pseudo:lem:estimate}.
\end{proof}

It is easy to check that the derivatives of $\Phi_\eps$ are bounded, so we can let
$
\overline{\Phi_\eps}
\colon
\overline{R_\eps} \times \coi{0, 2 \pi} \times \ooi{0, \pi}
\to
\real
$
denote the unique continuous extension of $\Phi_\eps$. In view of the previous result, we obtain the following expansion of
$\overline{\Phi_\eps}$.

\begin{cor}
\label{existence:cor:estimate_reduced}
There exist $\eps_0 > 0$ and
$\mu \in \oci{\parens{\lambda + 1} / \parens{\alpha + 1}, 1}$
such that
\begin{multline*}
\left|
	\overline{\Phi_\eps} \parens{r, \theta, \varphi}
	-
	K C_0
	-
	C_1 \parens*{
		\sum_{1 \leq j \leq K}
		V_\eps \parens*{P \parens{r, \theta, \varphi_j}}
	}
	-
\right.
\\
\left.
	-
	C_1^2 \eps^3
	\parens*{
		K
		+
		2
		\sum_{1 \leq j < k \leq K}
			\kappa_\eps \parens{r d_{j, k}}
	}
\right|
\lesssim
\eps^{3 + \mu}
\end{multline*}
for every $\eps \in \ooi{0, \eps_0}$ and
$
\parens{r, \theta, \varphi}
\in
\overline{R_\eps} \times \coi{0, 2 \pi} \times \cci{0, \pi}
$.
\end{cor}
\begin{proof}
It is easy to check that
\[
I_\eps'' \parens{W_{r, \theta, \varphi}} \brackets{u_1, u_2}
\lesssim
\parens{1 + \eps^3} \norm{u_1}_{H^1} \norm{u_2}_{H^1}
\]
for every $\eps > 0$,
$
\parens{r, \theta, \varphi}
\in
\ooi{0, \infty} \times \coi{0, 2 \pi} \times \cci{0, \pi}
$
and $u_1, u_2 \in H^1$. As such, we can use a Taylor expansion together with Lemmas \ref{pseudo:lem:estimate} and \ref{solving:lem:solution_auxiliary} to prove that
\[
\abs*{
	\overline{\Phi_\eps} \parens{r, \theta, \varphi}
	-
	I_\eps \parens{W_{r, \theta, \varphi}}
}
\lesssim
\eps^4
\]
for every
$\eps \in \ooi{0, \eps_0}$
and
$
\parens{r, \theta, \varphi}
\in
\overline{R_\eps} \times \coi{0, 2 \pi} \times \cci{0, \pi}
$,
so the result follows from Lemma \ref{reduced:lem:expansion_of_Ieps}.
\end{proof}

Let us use the previous expansion to prove that $\Phi_\eps$ has a minimum point, hence a critical point.

\begin{lem}
\label{reduced_functional:lem:existence_of_minimum}
We can fix $\eps_0 > 0$ such that given $\eps \in \ooi{0, \eps_0}$, there exists
$
\parens{r_\eps, \theta_\eps, \varphi_\eps}
\in
R_\eps \times \coi{0, 2 \pi} \times \cci{0, \pi}
$
such that
$
\Phi_\eps \parens{r_\eps, \theta_\eps, \varphi_\eps}
=
\inf \Phi_\eps
$.
\end{lem}
\begin{proof}
Disconsider the notation \eqref{existence:eqn:translation_of_varphi} exclusively in this proof. Let $\eps_0 > 0$ be such that the conclusions of Lemma \ref{solving:lem:solution_auxiliary} and Corollary \ref{existence:cor:estimate_reduced} hold. Associate each $\eps \in \ooi{0, \eps_0}$ to a
$
\parens{r_\eps, \theta_\eps, \varphi_\eps}
\in
\overline{R_\eps} \times \coi{0, 2 \pi} \times \cci{0, \pi}
$
such that
\[
\overline{\Phi_\eps} \parens{r_\eps, \theta_\eps, \varphi_\eps}
=
\min \overline{\Phi_\eps}
=
\inf \Phi_\eps,
\]
which is possible because $\overline{\Phi_\eps}$ is continuous and
$\overline{R_\eps} \times \mathbb{S}^1 \times \cci{0, \pi}$ is compact.

We claim that, up to shrinking $\eps_0$,
$
\parens{r_\eps, \theta_\eps, \varphi_\eps}
\in
R_\eps \times \coi{0, 2 \pi} \times \cci{0, \pi}
$
for every $\eps \in \ooi{0, \eps_0}$. By contradiction, suppose that we can fix
$\set{\eps_n}_{n \in \nat} \subset \ooi{0, \eps_0}$
such that
$\eps_n \to 0$ as $n \to \infty$ and
\[
\parens{r_n, \theta_n, \varphi_n}
:=
\parens{r_{\eps_n}, \theta_{\eps_n}, \varphi_{\eps_n}}
\in
\partial \parens{R_{\eps_n}} \times \coi{0, 2 \pi} \times \cci{0, \pi}
\]
for every $n \in \nat$. Due to the Pigeonhole Principle, one of the following holds up to subsequence:
\begin{enumerate}
\item \label{Case1}
$r_n = \eps_n^{- \frac{\alpha}{\alpha + 1}}$ for every $n \in \nat$;
\item \label{Case2}
$
r_n = \eps_n^{\parens*{\beta - \frac{\alpha - \lambda}{\alpha + 1}}}
$
for every $n \in \nat$;
\item \label{Case3}
given $n \in \nat$,
$
V \parens*{\eps_n r_n P_n}
=
1 + \eps^{\frac{3 \alpha}{\alpha + 1}}
$
up to replacing $\parens{\theta_n, \varphi_n}$, where
$P_n := P \parens{1, \theta_n, \varphi_n}$.
\end{enumerate}

\emph{Case \ref{Case1}.}
It follows from \ref{V_2} that
\[
V \parens{\eps_n^{\frac{1}{\alpha + 1}} P_n}
=
1
+
g \parens{\eps_n^{\frac{1}{\alpha + 1}} P_n}^\alpha
\leq
1 + \eps_n^{\frac{3 \alpha}{\alpha + 1}},
\]
for every $n \in \nat$, so
\[
\limsup_{n \to \infty}
\frac{1}{\eps_n^{\frac{1}{\alpha + 1}}}
\cdot
\frac{
	\nabla g \parens{\eps_n^{\frac{1}{\alpha + 1}} P_n}
	\cdot
	P_n
}{3 \eps_n^{\frac{1}{\alpha + 1}}}
\leq
1
\]
which contradicts the fact that $g'' \parens{0}$ is positive-definite.

\emph{Before cases \ref{Case2} and \ref{Case3}.}
Due to \ref{V_2} and Corollary \ref{existence:cor:estimate_reduced}, there exists $C_2 > 0$ and $n_0 \in \nat$ such that given
$n \geq n_0$,
\begin{multline}
\label{eqn:estimate}
\overline{\Phi_{\eps_n}} \parens{r_n, \theta_n, \varphi_n}
\leq
\Phi_{\eps_n} \parens{\rho_n, \theta_n, \varphi_n}
\leq
\\
\leq
K C_0
+
K C_1 \parens*{
	1
	+
	\eps_n^{\frac{\alpha \parens{\lambda + 1}}{\alpha + 1}}
}
+
C_1^2 \eps_n^3
\parens*{
	K^2 - \delta \eps_n^{\frac{\lambda + 1}{\alpha + 1}}
}
+
C_2 \eps_n^{3 + \mu},
\end{multline}
where
$
\rho_n
:=
\eps_n^{- \frac{\alpha - \lambda}{\alpha + 1}}
$
and
$\delta := \sum_{1 \leq j < k \leq K} \abs{d_{j, k}}$.

\emph{Case \ref{Case2}.}
Similarly, there exist $C_3 > 0$ and $n_1 \geq n_0$ such that
\begin{multline} \label{eqn:estimate_Case2}
K \parens{C_0 + C_1}
+
C_1^2 \eps_n^3 \parens*{
	K^2
	-
	\delta \eps_n^{\parens*{\beta + \frac{\lambda + 1}{\alpha + 1}}}
}
+
C_3 \eps_n^{3 + \mu}
\leq
\\
\leq
\overline{\Phi_{\eps_n}} \parens{r_n, \theta_n, \varphi_n}
=
\inf \Phi_{\eps_n}
\end{multline}
for every $n \geq n_1$. In view of \eqref{eqn:estimate} and \eqref{eqn:estimate_Case2},
\[
0
<
C_1^2 \delta \eps_n^{\parens*{3 + \frac{\lambda + 1}{\alpha + 1}}}
\leq
K C_1 \eps_n^{\frac{\alpha \parens{\lambda + 1}}{\alpha + 1}}
+
C_1^2
\delta
\eps_n^{\parens*{3 + \beta + \frac{\lambda + 1}{\alpha + 1}}}
+
\parens{C_2 - C_3} \eps_n^{3 + \mu}
\]
for every $n \geq n_1$, which is not possible because
$\lambda > 2 \parens{\alpha + 2} / \parens{\alpha - 1}$.

\emph{Case \ref{Case3}.}
There exist $C_3 > 0$ and $n_1 \geq n_0$ such that
\begin{multline}\label{eqn:estimate_Case3}
K C_0
+
K C_1
+
C_1 \eps_n^{\frac{3 \alpha}{\alpha + 1}}
+
C_1^2 \eps_n^3 \parens*{K^2 - \delta \eps_n^{\frac{1}{\alpha + 1}}}
+
C_3 \eps_n^{3 + \mu}
\leq
\\
\leq
\overline{\Phi_{\eps_n}} \parens{r_n, \theta_n, \varphi_n} = \inf \Phi_{\eps_n}
\end{multline}
whenever $n \geq n_1$. Due to \eqref{eqn:estimate} and \eqref{eqn:estimate_Case3}, we obtain
\[
0
<
C_1 \eps_n^{\frac{3 \alpha}{\alpha + 1}}
+
C_1^2 \delta \eps_n^{\parens*{3 + \frac{\lambda + 1}{\alpha + 1}}}
\leq
K C_1 \eps_n^{\frac{\alpha \parens{\lambda + 1}}{\alpha + 1}}
+
C_1^2 \delta \eps_n^{\parens*{3 + \frac{1}{\alpha + 1}}}
+
\parens{C_2 - C_3} \eps_n^{3 + \mu}
\]
for every $n \geq n_1$, which is not possible because $\lambda > 2$.
\end{proof}

It is only at this point that the necessity of considering pseudo-critical points whose peaks are placed according to specific geometric configurations becomes apparent.

\begin{rmk}
\label{reduced:rmk:impossible}
Suppose that we considered pseudo-critical points whose peaks are placed according to the following set (obtained as a naïve generalization of $\Lambda_\eps$ in \cite[Section 3]{ruizClusterSolutionsSchrodingerPoissonSlater2011}):
\begin{multline*}
\Lambda_\eps = \left\{
	\zeta \in \parens{\real^3}^K:
	\text{given}~j \neq k~\text{in}~\set{1, \ldots, K},
\right.
\\
\left.
	\frac{\eps^\beta}{\eps^{\frac{\alpha - \lambda}{\alpha + 1}}}
	<
	\abs{\zeta^j - \zeta^k}
	<
	\frac{1}{\eps^{\frac{\alpha}{\alpha + 1}}}
	\quad \text{and} \quad
	V \parens{\eps \zeta^j}
	<
	1 + \eps^{\frac{3 \alpha}{\alpha + 1}}
\right\}.
\end{multline*}
In this situation, Case \ref{Case2} would instead read as
\begin{enumerate}
\setcounter{enumi}{1}
\item
there exist $j \neq k$ in $\set{1, \ldots, K}$ such that
$
\abs{\zeta^j - \zeta^k}
=
\eps_n^{\parens*{\beta - \frac{\alpha - \lambda}{\alpha + 1}}}
$
for every $n \in \nat$.
\end{enumerate}
Once again, there exist $C_3 > 0$ and $n_1 \geq n_0$ such that
\begin{multline*}
K \parens{C_0 + C_1}
+
K C_1^2 \eps_n^3
+
2 C_1^2 \eps_n^3
\parens*{
	1
	-
	\frac{1}{2}
	\eps_n^{\parens*{\beta + \frac{\lambda + 1}{\alpha + 1}}}
}
+
\\
+
\parens*{K \parens{K - 1} - 2} C_1^2 \eps_n^3
\parens*{
	1
	-
	\frac{1}{2} \eps_n^{\frac{1}{\alpha + 1}}
}
+
C_3 \eps_n^{3 + \mu}
\leq
\overline{\Phi_{\eps_n}} \parens{r_n, \theta_n, \varphi_n}
=
\inf \Phi_{\eps_n}
\end{multline*}
for every $n \geq n_1$ and we can only continue the argument as before if $K = 2$.
\end{rmk}

\subsection{Proof of Theorem \ref{intro:thm}}
Fix $\eps_0 > 0$ such that Lemmas \ref{solving:lem:solution_auxiliary}, \ref{lem:natural_constraint} and \ref{reduced_functional:lem:existence_of_minimum} hold. In view of Lemma \ref{reduced_functional:lem:existence_of_minimum}, we can fix
\[
\set*{
	\parens{r_\eps, \theta_\eps, \varphi_\eps}
	\in
	R_\eps \times \coi{0, 2 \pi} \times \cci{0, \pi}:
	\quad
	\Phi_\eps \parens{r_\eps, \theta_\eps, \varphi_\eps}
	=
	\inf \Phi_\eps
}_{0 < \eps < \eps_0}.
\]
Given $\eps \in \ooi{0, \eps_0}$, let
\[
w_\eps
=
W_{r_\eps, \theta_\eps, \varphi_\eps}
+
n_{\eps, r_\eps, \theta_\eps, \varphi_\eps},
\]
so that $\nabla I_\eps \parens{w_\eps} = 0$ due to Lemma \ref{lem:natural_constraint}. The limit
$\norm{w_\eps - W_{r_\eps, \theta_\eps, \varphi_\eps}}_{H^1} \to 0$
as $\eps \to 0$ holds due to Lemma \ref{solving:lem:solution_auxiliary}, while the other limits follow by construction.

\qed

\sloppy
\printbibliography
\end{document}